\documentclass[10pt]{article}
\usepackage{amsmath}
\usepackage{amssymb}
\usepackage{indentfirst}
\usepackage[latin1]{inputenc}
\usepackage{geometry}
\usepackage{amsfonts}

\newcommand{\bc}{\begin{center}}
\newcommand{\ec}{\end{center}}

\newcommand{\qed}{\enspace\vrule  height6pt  width4pt  depth2pt}
\newenvironment{proof}{\par\noindent{\bf Proof.}}{$\qed$\par\bigskip}
\newtheorem{theorem}{Theorem}[section]

\newtheorem{lemma}[theorem]{Lemma}
\newtheorem{corollary}[theorem]{Corollary}

\begin{document}

\title{Hyperbolicity of Semigroup Algebras
II\thanks{Mathematics subject Classification Primary
[$16U60$, $20M25$]; Secondary [$20M10$, $16S36$].
Keywords and phrases: Semigroup, Semigroup Algebras,
Hyperbolic Groups, Group Rings, Units.\newline This
research is supported in part by Onderzoeksraad of
Vrije Universiteit Brussel, Fonds voor
Wetenschappelijk Onderzoek (Flanders) and
Flemish-Polish bilateral agreement BIL2005/VUB/06.}}
\author{E. Iwaki \and E. Jespers \and S. O. Juriaans \and A. C. Souza Filho}
\date{}
\maketitle

\begin{abstract}
In 1996 Jespers and Wang classified finite
semigroups whose integral semigroup ring has 
finitely many units. In a recent paper,
Iwaki-Juriaans-Souza Filho continued this line of
research by partially classifying the finite
semigroups whose rational semigroup algebra
contains a ${\mathbb{Z}}$-order with hyperbolic unit
group. In this paper we complete this
classification by handling the case in which the
semigroup is semi-simple.
\end{abstract}


\section{Introduction}

In this paper we continue the investigations on the
hyperbolic property. Recall (\cite{ijsf}) that a
unital $\mathbb{Q}$-algebra $A$ is said to have the
hyperbolic property if the unit group
${\mathcal{U}}(\Gamma )$ of every $\mathbb{Z}$-order
$\Gamma$ in $A$ does not contain a free abelian
subgroup ${\mathbb{Z}}^2$ of rank $2$.  If $A$ has
finite ${\mathbb{Q}}$-dimension, then having the
hyperbolic property is equivalent  to the unit group
of ${\mathbb{Z}}$-orders being hyperbolic groups in
the sense of Gromov \cite{gromov} and it is
sufficient to verify the required condition on only
one $\mathbb{Z}$-order in $A$.

The finite groups $G$ for which its rational group
algebra $\mathbb{Q}G$ is hyperbolic have been
characterized in \cite{jsp}. In \cite{jpp} this
problem has been dealt with for rational group
algebras of arbitrary groups and in \cite{jpsf} one
deals with groups algebras of finite groups over
quadratic extensions of the rationals. A natural
question is to extend these results to the much
wider class of (unital) rational semigroup algebras
of finite semigroups. A first step in this direction
is given in \cite{gdf} where the finite semigroups
$S$ are classified whose integral semigroup ring
$\mathbb{Z}S$ has trivial units. In a recent paper
(\cite{ijsf}), Iwaki-Juriaans-Souza Filho classified
the finite semigroups $S$ such that ${\mathbb{Q}} S$
has the hyperbolic property, this provided that
either $\mathbb{Q}S$ is semi-simple or $S$ is not a
semi-simple semigroup. See also \cite{jposf} and
\cite{thss} for results in the same direction.

In this paper we complete the classification started
in \cite{ijsf}. In particular, we need to handle the
case in which $S$ is semi-simple, that is, every
principal factor of $S$ is simple or $0$-simple and
$\mathbb{Q}S$ is not semi-simple.  The reader is
referred to \cite{clfrd,oknbook} for background on
semigroups and semigroup algebras. For convenience
sake, we recall that the contracted semigroup
algebra of a semigroup with zero element $\theta$ is
denoted by $\mathbb{Q}_{0}S$ and defined as
$\mathbb{Q}S/\mathbb{Q}\theta$. Since any semigroup
algebra can be considered as a contracted semigroup
algebra, we will state our results in this more
general context. Without specifically stating this,
throughout the paper we assume that
$\mathbb{Q}_{0}S$ is unital. Also recall that a
finite $0$-simple semigroup is isomorphic with a
Rees matrix semigroup
${\mathcal{M}}^{0}(G^{\theta},m,n;P)$, where $n$ and
$m$ are positive integers and $G$ is a finite group.
This is by definition the set of all $m\times
n$-matrices over $G^{\theta}$, the group $G$ with a
zero $\theta$ adjoined, with at all entries in
$G^{\theta}$ and at most one entry different from
$\theta$. The matrix $P$ is an $n\times m$-matrix
with entries in $G^{\theta}$ and it is regular
(i.e., each row and column contains at least one
non-zero element). The product of $A,B\in
{\mathcal{M}}^{0}(G^{\theta},m,n;P)$ is $A\circ P
\circ B$, where $\circ$ denotes the usual matrix
product.
The Jacobson radical of a ring $R$ we denote by
$J(R)$. An example that will be of importance in our
classification is the Rees matrix semigroup
$T={\mathcal{M}^{0}} (\{ 1 \}, 1, 2,P)$ with
$P=\left(
\begin{array}{l}
1 \\ 1
\end{array}
\right)$ . Clearly, $T=\{ e,f, \theta\}$ with
$ef=f$, $e^{2}=e$, $f^{2}=f$, $fe=e$,
$\theta^{2}=\theta$, and $e\theta = \theta e
=f\theta =\theta f =\theta$. By $T^{1}$ we denote
the smallest monoid containing $T$, thus
$T^{1}=T\cup \{ 1\}$, with $t1=1t$ for all $t\in T$.
Clearly, $J(\mathbb{Q}_{0}T^{1})=\mathbb{Q}(e-f)$
and
$\mathbb{Q}_{0}T^{1}=\mathbb{Q}(e-f)+\mathbb{Q}(1-e)+\mathbb{Q}(e)\cong
T_{2}(\mathbb{Q})$, the upper triangular $2\times
2$-matrices over $\mathbb{Q}$. So the semigroup
algebra $\mathbb{Q}_{0}T^{1}$ has the hyperbolic
property. We will show, in some sense, that this is
precisely the only case not covered by the main
result of \cite{ijsf}.

\section{Hyperbolic semigroup algebras}

For the convenience of the reader we recall the main result in \cite{ijsf}
on the structure of finite dimensional algebras with the hyperbolic property.

\begin{theorem}
\cite[Theorem $3.1$]{ijsf} \label{tfadfh} A finite
dimensional ${\mathbb{Q}}$ -algebra $A$ has the
hyperbolic property if and only if one of the
following holds:
\begin{enumerate}
\item $J(A)=\{ 0\}$ and $A =\oplus_{i} D_{i} \oplus B$,
\item $\dim_{{\mathbb{Q}}} J(A)=1$ and $A/J(A) =\oplus_{i} D_{i}$,
\end{enumerate}
where each $D_{i}$ is either $\mathbb{Q}$, a quadratic imaginary extension
of ${\mathbb{Q}}$ or is a totally definite quaternion algebra over
the rationals and where either $B \in \{ \{ 0\}, M_{2}(\mathbb{Q})\}$ or $B$
has a ${\mathbb{Z}}$-order whose unit group is non-torsion and hyperbolic.

In case $A$ has the hyperbolic property with
$J(A)\neq \{ 0\}$ then either $J(A)$ is central in
$A$ or $A$ is a direct product of division algebras
and $T_{2} ({\mathbb{Q}})$, the $2\times 2$-upper
triangular matrices over ${\mathbb{Q}}$.
\end{theorem}

Note that required conditions on the division algebras $D_{i}$ say that the
unit group of an order in $D_{i}$ has to be finite.

An obvious consequence of the Theorem is that ideals $I$ of finite
dimensional hyperbolic rational algebras $A$ also have an algebraic
structure as described in 1. or 2. of Theorem~\ref{tfadfh}. Here we consider
$I$ as a $K$-algebra, but it does not necessarily contain an identity.
Abusing terminology, we also will say that such algebras are hyperbolic.
Since these facts are crucial for the proof of our description of hyperbolic
rational semigroup algebras we formulate them in a corollary.

\begin{corollary}
\label{homimage} Let $A$ be a finite dimensional
rational algebra. If $A$ is hyperbolic then
epimorphic images and ideals of $A$ are hyperbolic.
In particular, if $J(I)=J(A)\cap I= \{ 0\}$ then $I$
has an identity and thus $A\cong I\times (A/I)$, a
direct product of algebras.
\end{corollary}


\begin{lemma}
\label{reesdescription} Let $G$ be a finite group
and $S$ a Rees semigroup
$S={\mathcal{M}^{0}}(G;m,n;P)$ with regular sandwich
matrix $P=(p_{i,j})$. If ${\mathbb{Q}}_0 S$ has the
hyperbolic property then $n m =1, 2$ or $4$.
Furthermore, if $nm\neq 1$ then $G=\{1\}$ and $nm=2$
or $n=m=2$ and $P\neq \left(
\begin{array}{ll}
1 & 1 \\ 1 & 1
\end{array}
\right)$. If $nm=2$ then $J(\mathbb{Q}_{0} S) \neq
\{ 0\}$ and $\mathbb{Q}_{0} S/J(\mathbb{Q}_{0}
S)\cong \mathbb{Q}$, and if $n=m=2$ then
$\mathbb{Q}_{0}S \cong M_{2}(\mathbb{Q})$.
\end{lemma}

\begin{proof}
Normalizing $P$, we may without loss of generality
assume that $p_{1,1}=1$. There is a natural
epimorphism $S={\mathcal{M}^{0}} (G;m,n;P)
\longrightarrow \overline{S}={\mathcal{M}^{0}}
(\{1\};m,n;{\overline{P}})$, where ${\overline{P}}$
is the matrix obtained from $P$ by replacing an
entry $g\in G$ by $1$. We thus obtain a natural
$\mathbb{Q}$-algebra epimorphism of contracted
semigroup algebras ${\mathbb{Q}}_{0} S
\longrightarrow {\mathbb{Q}}_{0}\overline{S}$.
Because of Corollary~\ref{homimage}, also the
algebra ${\mathbb{Q}}_{0} S$ has the hyperbolic
property. Recall the well known fact that
${\mathbb{Q}}_0{\mathcal{M}^{0}} (\{
1\};m,n;\overline{P})\cong
{\mathcal{M}}(\mathbb{Q},n,m;\overline{P})$, a Munn
algebra over the rationals. Hence (see
\cite[Theorem~3.6]{lp} or
\cite[Lemma~5.21]{oknbook}),
$\mathbb{Q}_{0}\overline{S}/J(\mathbb{Q}_{0}
\overline{S}) \cong M_{2}(\mathbb{Q})_{t}$, where
$t$ is the rank of $\overline{P}$. Because of
Theorem~\ref{tfadfh}, we then also know that $t=1$
if the radical is non-trivial and $t=2$ otherwise.
As also $\dim_{\mathbb{Q}}J(\mathbb{Q}_{0}
\overline{S})\leq 2$ we thus obtain that $mn-1= 1$
or $n m =4$ respectively. The former yields $mn=2$.

Now, as usual, we denote a non-zero matrix in $S$ as
$(g,i,j)$, where $g\in G $ is the only non-zero
entry in position $(i,j)$. All matrices of the type
$(\theta , i,j)$ are identified with the zero
element (also denoted by $\theta $) of $S$. Hence
the product in $S$ is as follows: $(g,i,j)\,
(h,k,l)=(gp_{j,k}h,i,l)$.

We now exclude the case that $m=1$ and $n=4$ (and
similarly, $m=4$ and $n=1$ is excluded). Suppose the
contrary, then $(1,1,1)-(1,1,2)$ and
$(1,1,1)-(1,1,3)$ are linearly independent nilpotent
elements of the hyperbolic algebra
$\mathbb{Q}_{0}(\overline{S})$. However, this is in
contradiction with the dimension of the Jacobson
radical being at most one.

If $n=m=2$ (and thus
$J(\mathbb{Q}_{0}(\overline{S}))=\{ 0\}$) then the
mapping $\mathbb{Q}_{0}\overline{S}\cong
{\mathcal{M}}(\mathbb{Q},2,2; \overline{P})
\longrightarrow M_{2}(\mathbb{Q}): A \mapsto A\circ
\overline{P}$ is an isomorphism (note that its
kernel consists of nilpotent elements). Hence $P\in
\mbox{GL}_{2}(\mathbb{Q})$. Therefore the regular
matrix can not be equal to $\left(
\begin{array}{ll}
1 & 1 \\ 1 & 1
\end{array}
\right)$. Hence $\overline{P}$ (and thus also $P$) is an upper or lower
triangular matrix. So $P$ is invertible over $\mathbb{Q}G$. Consequently the
mapping $\mathbb{Q}_{0}S\cong {\mathcal{M}}(\mathbb{Q}G,2,2;P)
\longrightarrow M_{2}(\mathbb{Q}G): A \mapsto A\circ \overline{P}$ is an
isomorphism. Theorem~\ref{tfadfh} then yields that $|G|=1$.

Finally, suppose that $nm=2$ and thus $t=1$. So all nilpotent elements
belong $J(\mathbb{Q}_{0}S)$. By symmetry, we then may assume that $n\geq 2$.
Let $g\in G$, $a=(gp_{2,1},1,1)$ and $b=(g,1,2)$. Then, because $p_{1,1}=1$,
$(a-b)^{2}=(gp_{2,1}gp_{2,1},1,1)+(gp_{2,1}g,1,2)-(gp_{2,1}g,1,2) -
(gp_{2,1}gp_{2,1},1,1)=0$. Consequently, $a-b\in J(\mathbb{Q}_{0}(S))$, for
all $g\in G$. Since $dim_{\mathbb{Q}} J(\mathbb{Q}_{0}(S))\leq 1$ this
implies that $|G|=1$.  This finishes the proof.
\end{proof}

Recall that a finite group is said to be a Higman
group if it is either abelian of exponent dividing
$4$ or $6$ or a Hamiltonian $2$-group. These are
precisely the finite groups $G$ so that
${\mathcal{U}}(\mathbb{Z}G)$ is finite (see for
example \cite{seh-book2}). By $C_{n}$ we denote the
cyclic group of order $n$. Put $S_{3}$ the symmetric
group of degree $3$, $D_{4}$ the dihedral group of
order $8$, $Q_{12}$ the quaternion group of order
$12$ and $C_{4}\rtimes C_{4}$ the semi-direct
product where $C_{4}$ acts non-trivially. These four
groups are precisely those non-abelian finite groups
$G$ so that $\mathbb{Q}G$ has the hyperbolic
property. It turns out that
${\mathcal{U}}(\mathbb{Z}G)$ has a free subgroup of
rank $2$ that is of finite index $2$ (\cite{jsp}).
From Dirichlet's Unit it follows that the only
finite abelian groups $H$ with ${\mathbb{Q}} H$
hyperbolic are $C_{5}$, $C_{8}$ and $C_{12}$. These
are the non-trivial finite abelian groups $H$ with
${\mathcal{U}}(\mathbb{Z}H)$ having an infinite
cyclic subgroup of finite index.  Also recall that a
null semigroup is a non-trivial semigroup with zero
in which the product of any two elements is the zero
element $\theta$. Of course such a semigroup is
$0$-simple if it has only one non-zero element.
Finally, if a principal factor of a finite semigroup
$S$ is of the form $G^{0}$, for some group $G$,
then, we simply say that this principal factor is
isomorphic with the group $G$.

We are now ready to prove our main result which,
together with \cite[Theorem $4.8$]{ijsf}, completes
the classification of the finite semigroups $S$
whose contracted semigroup algebra $\mathbb{Q}_{0}
S$ has the hyperbolic property. As said in the
introduction, only the  case of semi-simple
semigroups has to be dealt with. However, because of
Lemma~\ref{reesdescription}, the proof works for all
finite semigroups. So for this reason and because of
complenetess' sake, we give a complete proof.

\begin{theorem}
\label{maintheorem} Let $S$ be a finite semigroup (with zero). The
contracted semigroup algebra ${\mathbb{Q}}_{0}S$ has the hyperbolic property if and only
if all principal factors of $S$, except possibly one, say $K$, are Higman
groups, and $K$ is isomorphic to one of the following simple semigroups:

\begin{enumerate}
\item a null semigroup.

\item ${\mathcal{M}^{0}} (\{ 1 \}, 1, 2,P)$, with $P=\left(
\begin{array}{l}
1 \\ 1
\end{array}
\right)$.

\item ${\mathcal{M}^{0}} (\{ 1 \}, 2, 2,P)$ with $P\in \left\{ \left(
\begin{array}{ll}
1 & 0 \\ 0 & 1
\end{array}
\right), \left(
\begin{array}{ll}
1 & 1 \\ 0 & 1
\end{array}
\right) \right\}$.

\item $C_{5},\; C_{8}$, or $C_{12}$

\item $S_{3},\; D_{4},\; Q_{12}$, or $C_{4}\rtimes C_{4}$.
\end{enumerate}
\end{theorem}

\begin{proof}
Let $S_{n} =\{ \theta \} \subset S_{n-1} \subset
\cdots \subset S_{1}=S$ be a principal series of
$S$. So, the Rees factors $S_{i}/S_{i+1}$, with
$1\leq i\leq n-1$, are the non-trivial principal
factors of $S$. Each of these either is a null
semigroup or a Rees Matrix semigroup
${\mathcal{M}}^{0}(G_{i},m_{i},n_{i};P_{i})$, with
$G_{i}$ a finite group.

First assume that $\mathbb{Q}_{0}(S)$ is semisimple.
It is well known (see for example \cite[Theorem
14.24]{oknbook}) that this occurs precisely when
$n_{i}=m_{i}$ and $P_{i}$ is invertible in
$M_{n_{i}}(\mathbb{Q}G)$, for each $1\leq i\leq n-1$
and it follows that $\mathbb{Q}_{0}S \cong
\oplus_{i=1}^{n-1} \mathbb{Q}_{0}(S_{i}/S_{i+1})$.
So, $\mathbb{Q}_{0}S$ is hyperbolic if and only if
each $\mathbb{Q}_{0}(S_{i}/S_{i+1})$ has the
hyperbolic property. From Theorem~\ref{tfadfh} and
Lemma~\ref{reesdescription} we obtain that all
$n_{i}\leq 2$ and at most one of them, say
$n_{i_{0}}$, equals $2$. Furthermore, if the latter
occurs then $S_{i_{0}}/S_{i_{0}+1}$ is of the type
described in 3. For all other indices $i$, we have
that $n_{i}=1$ and $\mathbb{Z}G_{i}$ must have
finitely many units. Hence $G_{i}$ is a Higman
group. If, on the other hand, all $n_{i}=1$ then all
$\mathbb{Z}G_{i}$ have finite unit group except
possibly one which (by the comments before the
Theorem) is then of the type given in 4. or 5. In
the former case ${\mathbb{Q}} S_{0}$ is a sum of
division rings. In the latter case there is a simple
component of the type $M_{2}(\mathbb{Q})$. This
finishes the proof if $\mathbb{Q}_{0}S$ is
semisimple.

Second, assume $\mathbb{Q}_{0}S$ has the hyperbolic
property but not semi-simple. Choose the smallest
index $i_0$ such that the ideal ${\mathbb{Q}
}_0S_{i_0+1}$ of ${\mathbb{Q}}_{0}S$ has no
nilpotent elements. Note that $i_{0}>1$. As
$\mathbb{Q}_{0}S_{i_{0}}$ is an ideal of
$\mathbb{Q}_{0}S$ with trivial Jacobson radical,
Corollary~\ref{homimage} yields that
$\mathbb{Q}_{0}S \cong \mathbb{Q}_{0}S_{i_{0}}
\oplus \mathbb{Q}_{0}(S/S_{i_{0}})$. Again using
Theorem~\ref{tfadfh} we obtain that
$\mathbb{Q}_{0}S_{i_{0}}$ is hyperbolic and a direct
sum of division algebras of which the unit group of
an order is finite. So, as above, the principal
factors of $S_{i_{0}}$ are all Higman groups. So
$\mathbb{Q}_{0}(S/S_{i_{0}})$ has the hyperbolic
property and is not semisimple. Hence, to prove the
necessity of the conditions, we may assume that
$i_{0}=n-1$.

Suppose $S_{n-1}$ is a null semigroup then
$\mathbb{Q}_{0}S_{n-1}\subseteq J(\mathbb{Q}_{0}S)$
and thus, because of the dimension of the radical,
we have equality. Again, by Theorem~\ref{tfadfh}, we
then also have that
$\mathbb{Q}_{0}S/\mathbb{Q}_{0}S_{n-1}\cong
\mathbb{Q}_{0}(S/S_{n-1})$ is semisimple and
hyperbolic with finitely many units in an order. So,
all principal factors $S_{i}/S_{i+1}$, with $1\geq
i>n-1$ are Higman groups, as desired.

Next, suppose $S_{n-1}$ is not a null semigroup.
Because of Lemma~\ref{reesdescription} we get that
$S_{n-1}$ is as described in 2. So
$\mathbb{Q}_{0}S_{n-1}/J(\mathbb{Q}_{0}S_{n-1})\cong
\mathbb{Q}$ and $\mathbb{Q}_{0}S/J(\mathbb{Q}_{0}S)
\cong \mathbb{Q} \oplus \mathbb{Q}_{0}S/S_{n-1}$.
Theorem~\ref{tfadfh} then again yields that
$\mathbb{Q}_{0}(S/S_{n-1})$ is a sum of division
rings with finitely many units in an order. So, all
principal factors $S_{i}/S_{i+1}$, with $1\leq i\leq
n-2$, are Higman groups. This finishes the proof of
the necessity.

Conversely, suppose that the factors of the
principal series $S_{n} =\{ \theta \} \subset
S_{n-1} \subset \cdots \subset S_{1}=S$ of $S$ are
of the types described and that
$S_{i_{0}}/S_{i_{0}+1}$ is the exceptional factor.
If this principal factor is of type 3. then
$\mathbb{Q}_{0}S $ is semisimple and the result
follows from the first part of the proof. So, assume
$S_{i_{0}}/S_{i_{0}+1}$ is of the type 1. or 2. Now,
$S_{i_{0}+1}$ is a union of groups and therefore
$\mathbb{Q}_{0}S_{i_{0}+1}$ is semisimple and thus
has an identity. Moreover, ${\mathbb{Q}}_{0}S\cong
{\mathbb{Q}}_{0}S_{i_{0}+1}\oplus
{\mathbb{Q}}_{0}(S/S_{i_{0}+1})$ and, by the
semisimple part of the proof, an order in
$\mathbb{Q}_{0}S_{i_{0}+1}$ has finitely many units.
So, we only need to verify that
$\mathbb{Q}_{0}(S/S_{i_{0}+1})$ is hyperbolic. By
the assumptions and Corollary~\ref{reesdescription},
we know that
$J(\mathbb{Q}_{0}(S_{i_{0}}/S_{i_{0}+1})$ is of
dimension one and
$\mathbb{Q}_{0}(S_{i_{0}}/S_{i_{0}+1})/J(\mathbb{Q}_{0}(S_{i_{0}}/S_{i_{0}+1}))\cong
\mathbb{Q}$. The assumptions also imply that
$\mathbb{Q}_{0}(S/S_{i})$ is a direct sum of
division rings. Hence,
$\mathbb{Q}_{0}(S/S_{i_{0}+1})/J(\mathbb{Q}_{0}(S_{i_{0}}/S_{i_{0}+1}))\cong
\mathbb{Q} \oplus \mathbb{Q}_{0}(S/S_{i_{0}+1})$ is
a direct sum of division rings with only finitely
many units in an order. Therefore, $\mathbb{Q}_{0}S$
has the hyperbolic property. 
\end{proof}

\newpage

\noindent Centro de Matem\'atica, Computa\c c\~ao e
Cogni\c c\~ao,\newline Universidade  Federal do
ABC,\newline Rua Catequese, 242, $3^{\circ}$ andar,
Santo  Andr\'e, \newline CEP 09090-400 -
Brasil\newline email: edson.iwaki@ufabc.edu.br

\vspace{.25cm}

\noindent Department of Mathematics, \newline Vrije
Universiteit Brussel,\newline Pleinlaan 2, 1050
Brussel, Belgium\newline email: efjesper@vub.ac.be

\vspace{.25cm}

\noindent Instituto de Matem\'atica e Estat\'\i
stica,\newline Universidade de S\~ao Paulo
(IME-USP),\newline Caixa Postal 66281, S\~ao
Paulo,\newline CEP  05315-970 - Brasil \newline
email: ostanley@ime.usp.br

\vspace{.25cm}

\noindent Escola de Artes, Ci\^encias e
Humanidades,\newline Universidade de S\~ao Paulo
(EACH-USP),\newline Rua Arlindo B\'ettio, 1000,
Ermelindo Matarazzo, S\~ao Paulo, \newline CEP
03828-000 - Brasil \newline email:
acsouzafilho@usp.br

\end{document}